\documentclass[a4paper,12pt]{amsart}

\usepackage{amsmath}
\usepackage{amssymb}
\usepackage{graphicx}
\usepackage{amscd}
\usepackage[all]{xy}
\usepackage{amsfonts}
\usepackage{mathrsfs}
\usepackage{ifthen}
\usepackage{amsthm}
\usepackage{leqno}

\newcommand\Z{{\mathbb Z}}
\newcommand\C{{\mathbb C}}

\newcommand\Ac{{\mathcal A}}

\newcommand\Cc{{\mathcal C}}
\newcommand\Dcal{{\mathcal D}}
\newcommand\Ec{{\mathcal E}}

\renewcommand\L{{\mathcal L}}
\newcommand\M{{\mathcal M}}

\newcommand\Oc{{\mathcal O}}

\newcommand\Tc{{\mathcal T}}
\newcommand\U{{\mathcal U}}
\newcommand\Stab{\operatorname{Stab}}

\newcommand\End{\operatorname{End}}
\newcommand\Endc{\operatorname{\mathcal E\it{nd}}}

\newcommand\Hom{\operatorname{Hom}}
\newcommand\Homc{\operatorname{\mathcal H\it{om}}}
\newcommand\Diffc{\operatorname{\mathcal D\it{iff}}}
\newcommand\Der{\operatorname{Der}}

\newcommand\Det{\operatorname{Det}}

\newcommand\di{\operatorname{dim}}

\newcommand\Tg{\operatorname{T}}
\newcommand\Hcoh{\operatorname{H}}

\newcommand\Tr{\operatorname{Tr}}

\newcommand\Kr{\operatorname{Kr}}

\newcommand\Res{\operatorname{Res}}

\newcommand\iso{\kern.35em{\raise3pt\hbox
{$\sim$}\kern-1.1em\to}\kern.3em}
\newcommand\Gr{\operatorname{Gr}}

\newcommand\Gl{\operatorname{Gl}}
\newcommand\SGl{\operatorname{SGl}}
\newcommand\gl{\operatorname{\mathfrak gl}}
\newcommand\G{\operatorname{G}}

\newcommand\Ide{\operatorname{Id}}

\newcommand{\sgl}{\mathfrak{sgl}}
\newcommand{\g}{\mathfrak{g}}
\newcommand{\lgl}{\mathfrak{gl}}
\newcommand{\lsl}{\mathfrak{sl}}
\theoremstyle{plain}
\newtheorem{thm}{Theorem}[section]

\newtheorem{prop}[thm]{Proposition}
\theoremstyle{definition}
\newtheorem{defin}[thm]{Definition}

\theoremstyle{remark}
\newtheorem{remark}[thm]{Remark}
\numberwithin{equation}{section}

\begin{document}

\title[Uniformization of the moduli space of pairs $(X,E)$]{Uniformization of the \\ moduli space of pairs consisting of \\
a curve and a vector bundle.}

\author[\quad G\'omez Gonz\'alez, Hern\'andez Serrano, Mu\~noz Porras and Plaza Mart\'{\i}n]{E. G\'omez Gonz\'alez \\ D. Hern\'andez Serrano \\ J. M. Mu\~noz Porras \\  F. J. Plaza
Mart\'{\i}n}

\address{Departamento de Matem\'aticas, Universidad de
Salamanca,  Plaza de la Merced 1-4
        \\
        37008 Salamanca. Spain.
        \\
         Tel: +34 923294460. Fax: +34 923294583}
\address{IUFFYM. Instituto Universitario de F\'{\i}sica Fundamental y Matem\'aticas, Universidad de Salamanca, Plaza de la Merced s/n\\ 37008 Salamanca. Spain.}
\date\today
\thanks{
\indent This work is partially supported by the research contracts
MTM2009-11393 of Ministerio de Ciencia e Innovaci\'{o}n and  SA112A07 of JCyL.  \\
%\indent {\it E-mail addresses}: esteban@usal.es, dani@usal.es, jmp@usal.es,
%fplaza@usal.es}
}
\email{esteban@usal.es}
\email{dani@usal.es}
\email{jmp@usal.es}
\email{fplaza@usal.es}

\begin{abstract}
This paper is devoted to the study of the uniformization of the moduli space of pairs $(X,E)$ consisting of an algebraic curve and a vector bundle on it. For this goal, we study the moduli space of $5$-tuples $(X,x,z,E,\phi)$, consisting of a genus $g$ curve, a point on it, a local coordinate, a rank $n$ degree $d$ vector bundle and a formal trivialization of the bundle at the point. A group acting on it is found and it is shown that it acts (infinitesimally) transitively on this moduli space and an identity between central extensions of its Lie algebra is proved. Furthermore, a geometric explanation for that identity is offered.
\end{abstract}
\maketitle

%{\small \tableofcontents}

\section{Introduction.}\qquad

Uniformization of geometric objects, which is of special mathematical relevance on its own,  also has significant consequences in other topics such as mathematical physics. Let us illustrate this  by mentioning a couple of cases. The Uniformization Theorem of Riemann Surfaces was a key ingredient in Segal's approach to CFT (see the notion of {\sl annuli} in~\cite{SegalCFT}). Another example is the construction of the moduli space of vector bundles on an algebraic curve as a double coset, which has been applied in a variety of problems such as a proof of the Verlinde formula (\cite{BL}) and the development of the geometric theory of conformal blocks (e.g.~\cite{FrBen2}).

The infinitesimal study of uniformization has also led to connections of moduli theory, integrable systems and representation theory. Indeed, if one finds a group acting on a moduli space such that the action is infinitesimally transitive, then the action of the Lie algebra can help us to the study the properties of that moduli space; further, a link with the representation theory of infinite Lie algebras also becomes apparent. This has been the case in theory of loop groups, Virasoro algebra, Kac-Moody algebras, etc. As instances of this fruitful approach let us cite the interplay between the KP hierarchy and the Schottky problem (\cite{Mu1, Shiota}) and the infinitesimal version of Mumford's formula (\cite{MP2}).

Following the ideas of the above digression, this paper is devoted to the study of the uniformization of the moduli space of pairs $(X,E)$ consisting of an algebraic curve and a vector bundle on it. A group acting on it is found and it is shown that it acts (infinitesimally) transitively on this moduli space and an identity between central extensions of its Lie algebra is proved. Furthermore, a geometric explanation for that identity is offered.

Let us briefly review the contents of the paper. Following the spirit of \cite{MP2}, the group $\SGl_{\C((z))}(V)$ of semilinear automorphisms of $V=\C((z))^n$ is considered. Infinitesimal study of this group shows that its Lie algebra, $\sgl_{\C((z))}(V)$, is isomorphic to the Lie algebra of first-order differential operators (with scalar symbol) $\Dcal^1_{\C((z))/\C}(V)$. The relevance of this group lies in the fact that the central extension of its Lie algebra associated with its action on the Sato Grassmannian is the semidirect product of an affine Kac-Moody algebra and the Virasoro algebra. It is worth mentioning that this Lie algebra has already appeared in the literature (e.g. \cite{MR939049,GO}).

Following this, a study of certain central extensions of that Lie algebra is carried out and, as the first main result of the paper, we demonstrate an explicit identity (see Theorem~\ref{thm:LocalMumVB2}) among the cocycles associated with these central extensions:
    \begin{equation}
    \label{eq:formulaIntro}
    c_{n,\beta}\,=\,
    \beta c_{n,1} + (1-\beta)c_{n,0} + 6n\beta(\beta-1)vir_1
    \end{equation}
which can be thought of as a generalization of the infinitesimal version of the Mumford formula for the case of the moduli space of pairs $(X,E)$ ($X$ being a curve and $E$ a rank $n$ vector bundle on it). It should be noted that the Lie algebra of the group $\SGl_{\C((z))}(V)$ is closely related to Atiyah algebras and $\mathcal W$-algebras (these algebras have appeared in various models of two-dimensional quantum field theory and integrable systems, see for example \cite{BeilinsonSchectman,FKRW, AFMO}).

The following section is devoted to offering a geometric description of the group $\SGl_{\C((z))}(V)$. Let $\U_g^{\infty}(n,d)$ denote the moduli space of $5$-tuples $(X,x,z,E,\phi)$, consisting on a genus $g$ curve, a point on it, a local coordinate, a rank $n$ degree $d$ vector bundle and a formal trivialization of the bundle at the point, respectively. Thanks to the techniques of the Krichever map and the Sato Grassmannian, the tangent space to $\U_g^{\infty}(n,d)$ is described in cohomological terms as well as in terms of the geometry of the Sato Grassmannian. Then, the second main result of the paper can be shown; namely, that the group $\SGl_{\C((z))}(V)$ acts on $\U_g^{\infty}(n,d)$ and that this action is locally transitive  or, in other words, that the space $\U_g^{\infty}(n,d)$ is infinitesimally a homogenous space for the group $\SGl_{\C((z))}(V)$ (see Theorem~\ref{thm:SGLgeneratesUinfty}).

Finally,  our third main result (see Theorem~\ref{thm:GeometricFormulaLineBundles}) provides a geometric explanation for  formula (\ref{eq:formulaIntro}). In fact, given a family of smooth curves without automorphisms and a relative semistable rank $n$ degree $d$ vector bundle on it,  we succeed at building line bundles on it such that they  coincide infinitesimally with the central extensions of \S\ref{loc:s:central} and such that a relation analogous to equation (\ref{eq:formulaIntro}) holds. This relation resembles that obtained in \cite{Sch}. Moreover, if there exists a universal vector bundle on the universal curve (\cite{MestranoRamanan}), then our construction can be applied to it and an identity on the Picard group of the moduli space of vector bundles on the universal curve is obtained (\cite{Kou}).

We shall work over the field $\C$ of complex numbers, although all results are valid over an arbitrary algebraically closed field of characteristic $0$. When no confusion arises, and for the sake of clarity, we shall deal with rational points (i.e.,  $\C$-valued points).

\section{Group of semilinear automorphisms of $\C((z))^n$.}\label{loc:s:SGl}\qquad

\subsection{The group and its Lie algebra.}\qquad

Let us write $V=\C((z))^n$  as $n$-copies of the field of Laurent series and $V_+=\C[[z]]^n$ as $n$-copies of the formal power series ring. Let $\Gl_{\C}(V)$ be the restricted linear group and $\Gl_{\C((z))}(V)$ the $\C((z))$-linear group of $V=\C((z))^n$ (see \cite{SW,PS}). Let us denote by $\G$ the formal group scheme of automorphisms of $\C$-algebras of $\C((z))$ (see \cite{MP2}).

\begin{defin}\label{gen:d:SGl}
 We define the group functor  $\SGl_{\C((z))}(V)$ of semilinear automorphisms of $V$ as the subfunctor of $\Gl_{\C}(V)$ and its rational points are $\C$-linear automorphisms:
    $$\gamma \colon V\iso V 
    \, , $$
for which there exists an automorphism of $\C$-algebras of $\C((z))$, $g\in \G$, satisfying:
    \begin{equation}\label{loc:e:SGldef}
      \gamma (z\cdot v)=g(z)\cdot \gamma(v)
    \, .
    \end{equation}
\end{defin}

Following the ideas of \cite{MP2}, the definition of this group functor for points with values in any $\C$-scheme can be given.
\begin{prop}\label{gen:p:SGl}
One has a canonical exact sequence of group functors:
    $$0\to \Gl_{\C((z))}(V) \to \SGl_{\C((z))}(V) \to \G \to 0\, .$$
Moreover:
    $$\SGl_{\C((z))}(V)=\Gl_{\C((z))}(V)\rtimes \G
    \, .
    $$
\end{prop}

%\emph{$\pi$ or $p$ or $p_n$, $p_V$? Mirar.}
\begin{proof}
Bearing in mind that $\G$ acts on $\Gl_{\C((z))}(V)$ by conjugation, the result follows from \cite[Chapter IV.6]{SB}. In particular, the composition law in $\SGl_{\C((z))}(V)$ is explicitly given by:
    $$(\gamma_1,g_1)\ast (\gamma_2,g_2)=(c_{g_2}(\gamma_1)\circ \gamma_2,g_1 \circ g_2)\,,$$
where $c_{g_2}(\gamma_1)=g_2^{-1}\circ \gamma_1 \circ g_2$ is the action of $\G$ on $\Gl_{\C((z))}(V)$ by conjugation.
\end{proof}

 Henceforth we shall denote by $\g$, $\gl_{\C((z))}(V)$ and $\sgl_{\C((z))}(V)$ the Lie algebras of $\G$, $\Gl_{\C((z))}(V)$ and $\SGl_{\C((z))}(V)$, respectively. Let us define $\Diffc^1_{\C((z))/\C}(V,V)$ as the space of differential operators of order $\leq 1$ from $V$ to $V$ over $\C((z))$, and let us consider the subspace $\Dcal^1_{\C((z))/\C}(V,V)$ of scalar differential operators (see \cite[Ch.16]{EGAIV}).

\begin{prop}\label{gen:p:LieSGl}
We have that:
    $$\sgl_{\C((z))}(V)=\Dcal^1_{\C((z))/\C}(V,V)$$
as Lie subalgebras of $\End_{\C} V$.
\end{prop}

\begin{proof}
Let $\C[\epsilon]/(\epsilon^2)$ be the ring of dual numbers. By definition, $\sgl_{\C((z))}(V)$
consists of $\C[\epsilon]/(\epsilon^2)$-linear automorphisms $\gamma$ of $V\oplus \epsilon V$ such that $\gamma_{|{\epsilon=0}}=\Ide$ and for which there exists a $\C[\epsilon]/(\epsilon^2)$-algebra automorphism:
    $$g\colon \C((z))\oplus \epsilon \C((z)) \iso \C((z))\oplus \epsilon \C((z))$$
satisfying $g_{|{\epsilon=0}}=\Ide$ and $\gamma(zv)=g(z)\gamma(v)$.

Since $\gamma$ is a $\C[\epsilon]/(\epsilon^2)$-linear automorphism, one can write $\gamma=\Ide+\epsilon \gamma_0$,
where $\gamma_0\in \End_{\C} V$. Similarly, $g$ being a $\C[\epsilon]/(\epsilon^2)$-algebra automorphism implies that $g=1+\epsilon g_0$, where $g_0\in \g \iso \Der_{\C}\big(\C((z))\big)$ (see \cite{MP2}). Now the condition:
    $$
    (\Ide+\epsilon \gamma_0)(zv)=(1+\epsilon g_0)(z)(\Ide +\epsilon \gamma_0)(v)
    $$
implies:
    \begin{equation}\label{loc:e:gamma0}
      \gamma_0(zv)=z\gamma_0(v)+g_0(z)v \, ,
    \end{equation}
that is, $\gamma_0 \in \Dcal^1_{\C((z))/\C}(V,V)$. Thus, we have obtained a $\C$-vector space isomorphism:
    \begin{align*}
     \sgl_{\C((z))}(V)&\iso \Dcal^1_{\C((z))/\C}(V,V)\\
    \Ide+\epsilon \gamma_0 & \mapsto \gamma_0 \,.
    \end{align*}

It remains for us to show that this is a Lie algebra isomorphism. Observe that by the very definition $\Dcal^1_{\C((z))/\C}(V,V)$ fits into the following exact sequence:
    \begin{equation}\label{loc:e:DV}
      0\to \End_{\C((z))}V \to \Dcal^1_{\C((z))/\C}(V,V)\xrightarrow{\sigma} \Der_{\C}\big(\C((z))\big)\to 0\,,
    \end{equation}
($\sigma$ being the symbol map). Note that $\sigma (\gamma_0)=g_0$.

Bearing in mind that $\Der_{\C}\big(\C((z))\big) = \C((z))\partial_z$ and that this sequence splits as a sequence of vector spaces, we can write down the elements of $\Dcal^1_{\C((z))/\C}(V,V)$ as $\gamma+ g\partial_z$, where $\gamma \in \End_{\C((z))}V$ and $g\in \Der_{\C}\big(\C((z))\big)$. However, the Lie bracket of $\Dcal^1_{\C((z))/\C}(V,V)$, which is canonically inherited from that of $\End_{\C} V$,  is given by:
    $$
    [\gamma_1+g_1\partial_z,\gamma_2+g_2\partial_z]=
    [\gamma_1,\gamma_2]+[g_1\partial_z,g_2\partial_z]+g_1\partial_z(\gamma_2)-g_2\partial_z(\gamma_1)\,.
    $$
Setting $\{L_r=z^{r+1}\partial_z\,|\,r\in \Z\}$ as a basis for $\Der_{\C}\big(\C((z))\big)$, and $\{E_{ij}^s\,|\,s\in \Z\,,\,i,j=1,\dots,n\}$ as a basis for $\End_{\C((z))}V$ (where $E_{ij}^s$ is a $n\times n$ matrix whose $(i,j)$-entry is $z^s$ and $0$ otherwise). The Lie bracket of $\Dcal^1_{\C((z))/\C}(V,V)$ is given by following rules (see e.g. \cite{KRa}):
    \begin{equation}\label{genhit:e:braD^1}
      \begin{aligned}
    		&[L_r,L_s]=(s-r)L_{r+s}\\
    &[E_{ij}^r,E_{kl}^s]=\delta_{jk}E_{il}^{r+s}-\delta_{li}E_{kj}^{s+r}\\
    &[L_r,E_{kl}^s]= s E_{kl}^{r+s}
      \end{aligned}
    \end{equation}
It is straightforward to check that these expressions coincide with those for the Lie bracket of $\sgl_{\C((z))}(V)$.
\end{proof}

\begin{remark}
  Notice that equation (\ref{loc:e:gamma0}) says that $\gamma_0$ can be viewed as a covariant derivative along the vector field $g_0(z)\in \Der_{\C}\big(\C((z))\big)$. This kind of structure is also considered in \cite[Section 1]{KRad}.
\end{remark}

\subsection{Central extensions.}\label{loc:s:central}\qquad

The aim of this section is to compute the cocycles associated with some central extensions defined by $\sgl_{\C((z))}(V)$. Some of these central extensions come from pullbacks of algebras of the Virasoro type, while others arise as intertwinements of Kac-Moody and Virasoro algebras.

We shall review some facts concerning the construction of a family of Virasoro algebras (\cite[Section 3.5]{MP2}). Let $\Gr(V)$ denote the infinite Grassmannian associated with $(V,V_+)$ (\cite{Sa,SW}). Its rational points correspond to the vector subspaces $W\subseteq V$ such that:
    $$W\cap V_+ \, \mbox{ and }\, V/W+V_+$$
are finite-dimensional vector spaces over $\C$.

The group $\Gl_{\C}(V)$ acts on $\Gr(V)$ and preserves the determinant bundle. Therefore, one has the canonical central extension induced by the determinant bundle:
    $$1\to \C^{\ast} \to \widehat \Gl_{\C}(V) \to \Gl_{\C}(V) \to 1$$
and the cocycle associated with this central extension is given by:
    $$c(g_1,g_2)=\det (\bar g_1 \circ (\overline{g_1\circ g_2})^{-1}\circ \bar g_2)\,,$$
where $\bar g_i$ are preimages of $g_i$.

Let $\Ide + \epsilon_i D_i$ be a $\C[\epsilon_i]/\epsilon_i^2$-valued point of $\Gl_{\C}(V)$. The very definition of the cocycle at the Lie algebra level yields the expression:
    \begin{equation}\label{gen:e:Liecocycle}
      c_{Lie}(D_1,D_2)=\Tr(D_1^{+-}D_2^{-+}-D_2^{+-}D_1^{-+})\,,
    \end{equation}
where $D_i^{+-}\colon V^+ \to V^-$ is induced by $\Ide + \epsilon_iD_i$ w.r.t. the decomposition $V\iso V^- \oplus V^+$ where  $V^-=z^{-1}\C[z^{-1}]^n$, $V^+=\C[[z]]^n$.

Fix an integer number $\beta$ and consider the $\C$-vector space $\C((z))(dz)^{\otimes \beta}$. There is an action of $\G$ on $\C((z))(dz)^{\otimes \beta}$ defined by:
    \begin{equation}\label{loc:e:Gdif}
     \begin{aligned}
      \mu_{\beta}\colon \G \times \C((z))(dz)^{\otimes \beta} & \to \C((z))(dz)^{\otimes \beta}\\
    (g(z),f(z)(dz)^{\otimes \beta})&\mapsto f\big(g(z)\big)g'(z)^{\beta}(dz)^{\otimes \beta}
     \end{aligned}
    \end{equation}
which induces an action on $\Gr(\C((z))(dz)^{\otimes \beta})$ (denoted again by $\mu_{\beta}$) verifying:
$$\mu_{\beta}(g(z))=g'(z)^{\beta}\mu_0(g(z))$$
and preserving the determinant line bundle. Therefore, we may consider the associated central extension:
    \begin{equation}\label{loc:e:centralGbeta}
      1\to \C^{\ast} \to \widehat \G_{\beta} \to \G \to 1\,.
    \end{equation}
As $\beta$ varies in $\Z$, the cocycles corresponding to these central extensions are as follows:
    \begin{equation}\label{locm:e:betaG}
      vir_{\beta}(L_r,L_s)=\delta_{r,-s}\big(\frac{r^3-r}{6}\big)(1-6\beta+6\beta^2)
    \end{equation}
(at the Lie algebra level). Notice that $\widehat \g_1$ is precisely the Virasoro algebra, and the formula:
    \begin{equation}\label{loc:e:betaG2}
      vir_{\beta}=vir_{1}\cdot(1-6\beta+6\beta^2)
    \end{equation}
is a local analogue of the Mumford formula, where $vir_1$ is the standard cocycle associated with the Virasoro algebra.

Note that each central extension:
$$1\to \C^{\ast} \to \widehat \G_{\beta} \to \G \to 1$$
can be pulled back to $\SGl_{\C((z))}(V)$, by the surjection $\SGl_{\C((z))}(V) \xrightarrow{p_n} \G$ (recall that $n=\di_{\C((z))}V$), yielding a central extension:
    \begin{equation}\label{loc:e:centralGnbeta}
      1\to \C^{\ast} \to \SGl_{\C((z))}(V)\times_{\G}\widehat \G_{\beta} \to \SGl_{\C((z))}(V) \to 1\,.
    \end{equation}
Let us denote by $vir_{n,\beta}$ the cocycle corresponding to this central extension induced at the Lie algebra level. By construction, one has that:
    \begin{equation}\label{loc:e:vir_nbeta}
      vir_{n,\beta}=p_n^{\ast}(vir_{\beta})\,.
    \end{equation}
More explicitly, the following formulae hold:
    \begin{equation}\label{loc:e:cocycleGltriv}
    \begin{aligned}
      &vir_{n,\beta}(L_r,L_s)=n\cdot \delta_{r,-s}\cdot \frac{s^3-s}{6}(1-6\beta+6\beta^2)\\
      &vir_{n,\beta}(E_{ij}^{r},E_{kl}^{s})=0\\
      &vir_{n,\beta}(L_r,E_{ij}^s)=0
    \end{aligned}
    \end{equation}
with the same notations as equations~(\ref{genhit:e:braD^1}). In particular, $vir_{n,\beta}=n\cdot vir_{1,\beta}$.

The next step is to find  a family of central extensions of $\SGl_{\C((z))}(V)$ that, at the Lie algebra level, intertwine the structure of both $\lgl_{\C((z))}(V)$ Kac-Moody and Virasoro algebras.

Similarly to the case of $\G$, one has a natural action of $\SGl_{\C((z))}(V)$ on $V_{n,\beta}=(\C((z))(dz)^{\otimes \beta})^n$ defined by:
    $$\mu_{n,\beta}\big(\gamma(z\cdot v)\big)=g'(z)^{\beta}\cdot \gamma(v)\,.$$

\begin{thm}\label{loc:t:detbeta}
The action $\mu_{n,\beta}$ induces an action of $\SGl_{\C((z))}(V)$ on $\Gr(V)$, which preserves the determinant line bundle, and therefore there exists a central extension:
    $$1\to \C^{\ast} \to \widehat \SGl_{\C((z))}^{\beta}(V) \to \SGl_{\C((z))}(V)\to 1$$
canonically associated with $\mu_{n,\beta}$.
\end{thm}

\begin{proof}
By Proposition \ref{gen:p:SGl}, $\SGl_{\C((z))}(V)$ is the subgroup of $\Gl_{\C}(V)$ generated by $\G$ and $\Gl_{\C((z))}(V)$. Using \cite[Theorem 2.2]{MP2} the first part of the statement follows. The existence of the central extension is a consequence of \cite[Theorem 2.3]{MP2}.
\end{proof}

Taking into account the expressions of equation (\ref{genhit:e:braD^1}), the Lie algebra structure of $\widehat \sgl_{\C((z))}^{\beta}(V)$ is governed by the following rules:
      \begin{align*}
    &[L_r,L_s]=(s-r)L_{r+s}+c_{n,\beta}(L_r,L_s)\\
    &[E_{ij}^r,E_{kl}^s]=\delta_{jk}E_{il}^{r+s}-\delta_{li}E_{kj}^{s+r}+c_{n,\beta}(E_{ij}^{r},E_{kl}^{s})\\
    &[L_r,E_{kl}^s]= s E_{kl}^{r+s}+c_{n,\beta}(L_r,E_{kl}^s)
      \, ,
      \end{align*}
where $c_{n,\beta}$ denotes the corresponding cocycle. Thus, it remains for us to compute this cocycle.

\begin{prop}\label{loc:p:betaSGl}
The cocycle, $c_{n,\beta}$, associated with the central extension $\widehat \sgl_{\C((z))}^{\beta}(V)$ is given by:
    \begin{align*}
      &c_{n,\beta}(L_r,L_s)=n\cdot \delta_{r,-s}\cdot \frac{r^3-r}{6}(1-6\beta+6\beta^2)\\
      &c_{n,\beta}(E_{ij}^{r},E_{kl}^{s})=\delta_{r,-s}\delta_{il}\delta_{jk}\cdot s\\
      &c_{n,\beta}(L_r,E_{ij}^s)=\delta_{r,-s}\delta_{ij}\cdot \frac{r(r+1)}{2}(1-2\beta)
    \end{align*}
\end{prop}

\begin{proof}
Let us denote with $\{e_k \,|\, k\in \Z\}$ a basis of $V$, where $e_k=(0,\dots ,0,z_{k_2}^{k_1},0,\dots ,0)$ (i.e. $z^{k_1}$ lying in the $k_2$-th entry and $0$ elsewhere) and $k=k_1n+k_2-1$ with $k_1\in \Z$ and $k_2=1,\dots ,n$.

Recall that the action of $\SGl_{\C((z))}(V)$ on $V_{n,\beta}=(\C((z))(dz)^{\otimes \beta})^n$ is defined by:
    $$\mu_{n,\beta}\big(\gamma(z\cdot v)\big)=g'(z)^{\beta}\cdot \gamma(v)\,,$$
Therefore, the action of $\Gl_{\C((z))}(V)$ does not depend on $\beta$. One has that the $\Z \times \Z$-matrix associated with $\mu_{n,\beta}(L_r)$ is:
    $$\big(\mu_{n,\beta}(L_r)\big)_{lk}=\left \{
    \begin{array}{ll}
    1 &\text{ if }  l_1=k_1\, \mbox{ and }\,l_2=k_2 \\
    \epsilon \cdot \big(k_1+\beta(1+r)\big) &\text{ if } l_1=k_1+r\, \mbox{ and } \,l_2=k_2 \\
    0 &\text{ otherwise, }
    \end{array}\right .$$
and the element in $\End_{\C}(V)$ corresponding to $(\Ide + \epsilon E_{ij}^r)$ is:
    $$(\Ide + \epsilon E_{ij}^{r})_{lk}=\left \{
    \begin{array}{ll}
    1  &\text{ if }  l_1=k_1\, \mbox{ and }\,l_2=k_2\\
    \epsilon  &\text{ if } l_1=k_1+r \,\mbox{ and }\, l_2=i \,\mbox{ and }\,k_2=j \\
    0  &\text{ otherwise. }
    \end{array}\right .$$
Taking into account equation (\ref{gen:e:Liecocycle}), the result follows.
\end{proof}

\begin{thm}\label{thm:LocalMumVB2}
The following relation holds:
    $$c_{n,\beta}= \beta c_{n,1} + (1-\beta)c_{n,0} + 6n\beta(\beta-1)vir_1 $$
\end{thm}

\begin{proof}
Using Proposition \ref{loc:p:betaSGl} and equation (\ref{loc:e:cocycleGltriv}), the values of the cocycles for a pair of elements of the basis can be arranged as seen in the following table:
    {\small \begin{equation*}
      \begin{array}{|c|c|c|c|}
    \hline
     & (L_r,L_s) & (L_r,E_{ij}^s) & (E_{ij}^r,E_{kl}^s) \\
    \hline
    vir_{n,1} & n\cdot \delta_{r,-s}\frac{r^3-r}{6} & 0 & 0 \\
    \hline
    c_{n,0} & n\cdot \delta_{r,-s}\frac{r^3-r}{6} &\delta_{r,-s}\delta_{ij}\cdot \frac{r(r+1)}{2} &\delta_{r,-s}\delta_{il}\delta_{jk}\cdot s \\
    \hline
    c_{n,1} & n\cdot \delta_{r,-s}\frac{r^3-r}{6} & -\delta_{r,-s}\delta_{ij}\cdot \frac{r(r+1)}{2} & \delta_{r,-s}\delta_{il}\delta_{jk}\cdot s\\
    \hline
    c_{n,\beta} & n\cdot \delta_{r,-s}\frac{r^3-r}{6}(1-6\beta+6\beta^2)& \delta_{r,-s}\delta_{ij}\cdot \frac{r(r+1)}{2}(1-2\beta) & \delta_{r,-s}\delta_{il}\delta_{jk}\cdot s\\
    \hline
    %(1-\beta)c_{n,0}+\beta c_{n,1} & n\cdot \delta_{r,-s}\frac{r^3-r}{6} & \delta_{r,-s}\delta_{ij}\cdot %\frac{r(r+1)}{2}(1-2\beta) & \delta_{r,-s}\delta_{ij}\cdot \frac{r(r+1)}{2}(1-2\beta)\\
    %\hline
    \end{array}
    \end{equation*}}
From this table, one notes that:
    $$c_{n,\beta}= \beta c_{n,1}+ (1-\beta)c_{n,0} +6\beta(\beta-1)vir_{n,1} \,,$$
and by virtue of equation (\ref{loc:e:vir_nbeta}) one has:
    $$vir_{n,1}=n\cdot vir_1\,.$$
Therefore:
    $$c_{n,\beta}=\beta c_{n,1} + (1-\beta)c_{n,0} +  6n\beta(\beta-1)vir_1 \,.$$
\end{proof}

Let us see how the above result can be restated in terms of bitorsors or, equivalently, line bundles over $\SGl_{\C((z))}(V)$ (see \cite[Expos\'{e} VII]{SGA7} for the relationships among bitorsors, line bundles and extensions). Let ${\mathbb L}_{n,\beta}$ denote the bitorsor over $\SGl_{\C((z))}(V)$ associated with $\widehat \SGl_{\C((z))}^{\beta}(V)$ (see Theorem~\ref{loc:t:detbeta}) and let $\Lambda_1$ be the bitorsor over $G$ corresponding to $\widehat G_1$ (equation~(\ref{loc:e:centralGbeta})). Thus, Theorem~\ref{thm:LocalMumVB2} is equivalent to the following identity:
    \begin{equation}\label{eq:TorsorsOnSGL}
    {\mathbb L}_{n,\beta}\, \simeq
    \, {\mathbb L}_{n,1}^{\otimes\beta} \otimes {\mathbb L}_{n,0}^{\otimes(1-\beta)} \otimes p^* \Lambda_1^{\otimes 6n\beta(\beta-1)}
    \, , 
    \end{equation}
where $p$ is the natural projection map $\SGl_{\C((z))}(V)\to G$.

\subsection{Comments on $\sgl_{\C((z))}(V)$: relationships with other Lie algebras.}\label{subsec:structSGL}\qquad

Let us finish this section with a brief discussion of the properties of $\SGl_{\C((z))}(V)$ and, more precisely, of the relationship between its Lie algebra and the Atiyah  and  $\mathcal W_{1+\infty}$ algebras. From our point of view,  these connections make  $\SGl_{\C((z))}(V)$ a relevant object that deserves deeper study.

To begin with, observe that there is a pullback map from the central extensions of $\sgl_{\C((z))}(V)$ to those of $\lgl_{\C((z))}(V)$:
    \begin{equation}\label{eq:restrictionH2}
    \Hcoh^2(\sgl_{\C((z))}(V),\C)\,\longrightarrow\, \Hcoh^2(\lgl_{\C((z))}(V),\C)
    \, , 
    \end{equation}
defined by restriction through the inclusion $\lgl_{\C((z))}(V)\hookrightarrow \sgl_{\C((z))}(V)$. Note that the Kac-Moody algebra is a central extension of $\lgl_{\C((z))}(V)$. In this way, we establish a link with the theory of Kac-Moody algebras and, by considering the action on the spaces of global sections of powers of the determinant bundle, one obtains semi-infinite wedge representations of such algebras (e.g. \cite[Lecture 9]{KRa}). Additionally, recall that section~\ref{loc:s:central} provides a map:
    $$\Hcoh^2(\g,\C)\hookrightarrow \Hcoh^2(\sgl_{\C((z))}(V),\C)\,,$$
(where $\g=Lie(G)$ is referred to in the literature as the Witt algebra and its central extension is the Virasoro algebra).

Summing up, the structures of both Kac-Moody and Virasoro algebras are intertwined naturally into a single object, namely, the group functor $\SGl_{\C((z))}(V)$. This group will be endowed with a geometric meaning (in terms of vectors bundles over algebraic curves) in the following section.

Furthermore, the Lie algebra $\sgl_{\C((z))}(V)$ can be thought of as a formal analogue to the so-called Atiyah algebras of \cite{BS}. It is worth pointing out that by the equivalence of categories between Atiyah algebras and algebras of differential operators (see  \cite{BS} \S1), the algebra of differential operators associated with $\sgl_{\C((z))}(V)$ is precisely $\lgl_{\C((z))}(V)\otimes_{\C} \C[[\partial_z]]$; that is, the algebra of differential operators of arbitrary order whose coefficients are matrices.

It is well known (\cite{LW}) that $\di \Hcoh^2(\lgl_{\C((z))}(V)\otimes_{\C} \C[[\partial_z]],\C)=1$ or, in other words, that it has essentially a unique central extension (the simplest case dates back to \cite{Rad}). Bearing in mind the explicit expressions for such $2$-cocycles given in \cite{KPe} and {\cite{KRa}}, one is able to describe such an extension explicitly for the case of $\lgl_{\C((z))}(V)\otimes_{\C} \C[[\partial_z]]$:
    $$
    \Psi(A(z)\partial_z^r, B(z)\partial_z^s) \,:=\,
    \frac{r! s!}{(r+s+1)!} \operatorname{Res}_{z=0} \Tr\big(\partial_z^{s+1}A(z)
    \cdot \partial_z^{r}B(z)\big)dz
    \, .
    $$
It is  remarkable that the restriction of this cocycle to the Lie subalgebra of differential operators of order $\leq 1$  coincides with the expression computed in Proposition \ref{loc:p:betaSGl} (for $\beta=0$).

Another relevant reference about the relationship between the representation theory of Virasoro and Kac-Moody algebras and quantum physics is \cite{GO}. The semidirect product of the Virasoro and Kac-Moody algebra considered there coincides with the Lie algebra $\widehat \sgl_{\C((z))}^{\beta}(V)$ for $\beta=\frac{1}{2}$ (see also formula $(4.2)$ of  \cite{Semi} for the rank one case).

Let us now restrict ourselves to the case of $\di_{\C((z))}V=1$; i.e. $V=\C((z))$. Thus, the Lie algebra corresponding to the above cocycle is called the ${\mathcal W}_{1+\infty}$-algebra and its representation theory in terms of vertex operator algebras has been  studied in depth (\cite{FKRW}, see also \cite{AFMO}).  In order to shed some light on this, let us note that the action on the Fock space:
    $$
    \widehat\SGl^1_{\C((z))}(\C((z))) \,\hookrightarrow \, \Gl\big( \Hcoh^0(\Gr(\C((z))),\Det^*)\big)
    $$
induces a map between their Lie algebras that extends to:
    $$
    {\mathcal W}_{1+\infty} \, \longrightarrow\, \End\big( \Hcoh^0(\Gr(\C((z))),\Det^*)\big)
    \, .
    $$

In this case, the vector space parametrizing central extensions of the Lie algebra $\sgl_{\C((z))}(\C((z)))$ has been explicitly computed in~\cite{ACKP}, where it was applied to compute some cohomology groups of moduli spaces. There, it was shown that this space is three-dimensional and that is generated by the following $2$-cocycles:
    \begin{equation*}
    \begin{gathered}
    \alpha_1(f_1\partial_z +g_1,f_2\partial_z +g_2) \, =\, \Res_{z=0}f_1 \partial_z^3 f_2 dz
    \\
    \alpha_2(f_1\partial_z +g_1,f_2\partial_z +g_2)\, =\,
    \Res_{z=0}(f_1 \partial_z^2 g_2-f_2 \partial_z^2 g_1)dz
    \\
    \alpha_3(f_1\partial_z +g_1,f_2\partial_z +g_2)\, =\,
    \Res_{z=0}g_1 \partial_z g_2 dz
    \end{gathered}
    \end{equation*}

The Lie algebra  $\sgl_{\C((z))}(\C((z)))$ is also presented in \cite[formula 16]{CD} as the algebra of asymptotic symmetries of the warped black-hole geometries.

Unfortunately, when $V$ is of arbitrary dimension, the whole group is not known. However, our Theorem~\ref{thm:LocalMumVB2} shows that:
    $$
    vir_{n,\beta}\, \in \,
    < vir_{n,1} , vir_{n,0},  p_n^*(vir)> \,\subseteq\,
    \Hcoh^2\big(\sgl_{\C((z))}(V),\C\big)
    $$
and provides the coefficients of the linear combination. Further research will be performed to compute whether $\Hcoh^2\big(\sgl_{\C((z))}(V),\C\big)\iso \C^{3}$, and to generalize (to higher rank case) the results of \cite{ACKP}.

Finally, we direct  interested readers to \cite{Schli} for a cohomological study of central extensions of Krichever-Novikov algebras, which in a certain sense, generalize the case of $\sgl_{\C((z))}(V)$.

\section{The moduli of curves with vector bundles: $\U_g^{\infty}(n,d)$.}\label{loc:s:U}\qquad

\subsection{The Krichever map and the representability of $\U_g^{\infty}(n,d)$}\qquad

Let us offer a quick overview of the relationship between the moduli space of vector bundles and the infinite Sato Grassmannian. The following statements are taken from \cite{Sa,SW,Mu,MP1}.

Let $\U_g^{\infty}(n,d)$ denote the moduli functor whose rational points are tuples $(X,x,z,E,\phi)$ where $X$ is a smooth, projective and irreducible curve of genus $g$ with  marking $x\in X$, $z$ is a formal parameter on $x$:
    \begin{equation}\label{gen:e:z}
    \widehat \Oc_{X,x}\iso \C[[z]]\,,
    \end{equation}
$E$ is a rank $n$ and degree $d$ vector bundle on $X$, and  $\phi$ is a $\widehat \Oc_{X,x}$-module isomorphism:
    \begin{equation}\label{gen:e:phi}
      \phi \colon\widehat E_x\iso \widehat \Oc_{X,x}^n\,.
    \end{equation}

Recall that the infinite Grassmannian of $(V=\C((z))^n, V_+=\C[[z]]^n)$ has a decomposition into connected components:
    $$\Gr(V)=\coprod_{m\in \Z}\Gr^m(V)\,,$$
where $\Gr^m(V)$ consists of the subspaces $W\in \Gr(V)$ such that:
$$m=\di_{\C}(W\cap V_+)-\di_{\C}(V/W+V_+)\,.$$

\begin{defin}\label{gen:d:Kr}
The Krichever map is defined by:
    \begin{equation*}
    \begin{aligned}
    Kr\colon \U_g^{\infty}(n,d)& \,\longrightarrow\,
    \Gr^{\chi}(V) \\
    (X,x,z,E,\phi) & \,\longmapsto\, \Hcoh^0(X-x,E)
    \, ,
    \end{aligned}
    \end{equation*}
where $\Hcoh^0(X-x,E)$ is understood as a subspace of $V$ via the isomorphisms (\ref{gen:e:z}) and (\ref{gen:e:phi}), and $\chi =n(1-g)+d$ (that is, the Euler-Poincar\'e characteristic of $E$).
\end{defin}

Its image is characterized by the following theorem:

\begin{thm}\label{gen:t:carU}
A point $W\in  \Gr^{\chi}\big(V\big)$ lies on the image of the Krichever map if the stabilizer algebra of $W$:
    $$A_W:=\Stab(W)=\{f\in \C((z))\,|\,f\cdot W\subseteq W\}$$
belongs to $\Gr^{1-g}\big(\C((z))\big)$ and $A_W$ is a regular ring. Moreover, the Krichever morphism is injective.
\end{thm}

\begin{proof}
This follows from \cite{Mu} and \cite{MP1}.
\end{proof}

\begin{remark}
Notice that the smoothness of $X$ is equivalent to the regularity of $A_W$. If $X$ is allowed to be singular, then the characterization requires a maximality condition for $A_W$ analogous to that of \cite[Section 6]{SW}, where the case of rank $1$ was considered. If one aims at characterizing points with values in any $\C$-scheme $S$, one can follow \cite{Mu}; in this case the flatness of $A_W$ over $S$ has to be imposed.
\end{remark}

\begin{thm}\label{gen:t:repU}
 The moduli space $\U_g^{\infty}(n,d)$ is representable by a subscheme of $\Gr^{\chi}(V)$.
\end{thm}
\begin{proof}
This follows from \cite{Mu} and \cite{MP1}.
\end{proof}

\subsection{Tangent space to $\U_g^{\infty}(n,d)$.}\label{gen:ss:tg}\qquad

With the same notations as before, let us now introduce $\Diffc_{X/\C}^1(E,E)$ as the sheaf of differential operators of order $\leq 1$ from $E$ to $E$ over $\Oc_X$. Recall that it fits into the following exact sequence of sheaves:
$$0\to \Endc_{\Oc_X}E \to \Diffc^1_{X/\C}(E,E)\xrightarrow{\sigma} \Tc_X\otimes_{\Oc_X}\Endc_{\Oc_X}E \to 0$$
where $\sigma$ is the symbol morphism (see \cite[Ch.16]{EGAIV}) and $\Tc_X$ denotes the tangent sheaf.

Let us now consider the subsheaf $\Dcal^1_{X/\C}(E,E)$ of $\Diffc^1_{X/\C}(E,E)$ of scalar differential operators, that is, those whose symbol morphism take values in:
    $$\Tc_X\otimes_{\Oc_X}\Oc_X\hookrightarrow \Tc_X\otimes_{\Oc_X}\Endc_{\Oc_X}E
    \, ,
    $$
where $\Oc_X\hookrightarrow \Endc_{\Oc_X}E$ is the canonical morphism. That is:
    \begin{equation}\label{gen:e:Dcal}
    0\to \Endc_{\Oc_X}E \to \Dcal^1_{X/\C}(E,E)\xrightarrow{\sigma} \Tc_X\to 0
    \, .
    \end{equation}
Note that if $E$ has rank one, then $\Diffc^1_{X/\C}(E,E)=\Dcal^1_{X/\C}(E,E)$.

\begin{remark}
In the literature, this sequence is also refereed to as the Atiyah exact sequence, and $\Dcal^1_{X/\C}(E,E)$ is the well-known Atiyah bundle (one can find a formal analogue of this sequence in equation \ref{loc:e:DV}). For the purpose of this paper, it is also interesting to think of this object as the Atiyah algebra of \cite{BS}. Furthermore, the bundle $\Ac_{\Tr E}$ of \cite{BS} is the extension:
    $$0\to \Oc_X \to \Dcal^1_{X/\C}(E,E)\xrightarrow{\sigma} \Tc_X\to 0$$
induced by the sequence (\ref{gen:e:Dcal}) and the trace map $\Tr \colon \Endc_{\Oc_X}E \to \Oc_{X}$.
\end{remark}

\begin{thm}\label{gen:t:tgUinfty}
One has an isomorphism of $\C$-vector spaces:
    $$\Tg_{\Ec} \U_g^{\infty}(n,d) \iso \varprojlim_m \Hcoh^1(X,\Dcal^1_{X/\C}(E,E(-mx)))
    \, .
    $$
where $\Ec=(X,x,z,E,\phi) \in \U_g^{\infty}(n,d)$ is a rational point.
\end{thm}

\begin{proof}
  Let $\U^m$ be the functor whose rational points are $\Ec_m=(X,x,z_m,E,\phi_m)$ where $X$ is a smooth projective curve, $x\in X$ is a point,
$$\xymatrix{ z_m \colon \Oc_X \ar@{->>}[r] & \C[z]/z^m \C[z] }$$
is a $m$-level structure at $x\in X$, $E$ is a rank $n$ vector bundle on $X$ and $\xymatrix{ \phi_m \colon E \ar@{->>}[r] & (\Oc_X/\Oc_X(-mx))^{\oplus n}}$ is a $m$-level structure of $E$ (see \cite{Se3}). Note that $\U_g^{\infty}(n,d)=\varprojlim_m \U^m$.

Recalling the ideas of \cite{ACKP,BiS,Wel,Li}, one has that:
$$\Tg_{\Ec_m} \U^m \iso \Hcoh^1(X,\Dcal^1_{X/\C}(E,E(-mx)))\,,$$
and taking the inverse limit in $m$, the result follows.
\end{proof}

Theorem \ref{gen:t:tgUinfty} can be also interpreted via the Krichever map, that is, one can explicitly identify $\Tg_{\Ec}\U_g^{\infty}(n,d)$ with a subspace of $\Tg_{\Ec}\Gr^{\chi}(V)$. Denote by $\M_g^{\infty}$ the moduli space of triples $(X,x,z)$, where $X$ has genus $g$. From (\cite{SW,MP3}) we know that:
    $$\Tg_W \Gr(V)\iso \Hom_{\C}(W,V/W)$$
    $$\Tg_{A_W} \M_g^{\infty}\iso \Der_{\C}(A_W,\C((z))/A_W)\,,$$
for $W\in \Gr(V)$ and $A_W\in \Gr\big(\C((z))\big)$.

\begin{prop}\label{gen:p:tgUGr}
Let $\Ec$ be a point in  $\U_g^{\infty}(n,d)$ and $W$ its image under the Krichever map. The Krichever map induces an isomorphism of $\C$-vector spaces from the tangent space of $\U_g^{\infty}(n,d)$ at $\Ec$ to the space of first order scalar differential operators from $W$ to $V/W$ as $A_W$-modules or, what is tantamount:
$$\Tg_{\Ec} \U_g^{\infty}(n,d) \iso \Dcal^1_{A_W/\C}(W,V/W)\,.$$
\end{prop}

\begin{proof}
Using theorem \ref{gen:t:carU} we can write:
    {\small
    $$\Tg_W\U_g^{\infty}(n,d)=\{\overline{W}\in \Tg_W\Gr(V) \mbox{ such that }A_{\overline{W}}=\Stab(\overline{W})\in \Tg_{A_W}\Gr\big(\C((z))\big)\}\,.
    $$}
Since $A_{\overline{W}}$ is a $\C[\epsilon]/(\epsilon^2)/\epsilon^2$-algebra, there exists $g\in \Der_{\C}(A_W,\C((z))/A_W)$ satisfying $A_{\overline{W}}=\{a+\epsilon g(a), \quad a\in A_W\}$. On the other hand, $\overline{W}\in \Tg_W\Gr(V)$, so there exists $f\in \Hom_{\C}(W,V/W)$ such that:
    $$\overline{W}=\{w+\epsilon f(w),\quad w \in W\}\,.$$

$\overline{W}$ being an $A_{\overline{W}}$-module, for each $a\in A_W$ and $w\in W$ there exists $w'\in W$ satisfying:
    $$(a+\epsilon g(a))(w+\epsilon f(w))=w'+\epsilon f(w')
    \, ,
    $$
from where we deduce:
    $$f(aw)=af(w)+g(a)w\,.$$
Thus, following \cite[Ch.16]{EGAIV}, $f$ belongs to $\Dcal^1_{A_W/\C}(W,V/W)$. The converse is straightforward.
\end{proof}

Let us say few words about the dependence between the infinitesimal deformation of the curve and the infinitesimal deformation of the bundle. Let $\U_g^{\infty}(n,d)\to \M_g^{\infty}$ be the forgetful functor. Therefore, its fiber at the point $(X,x,z)$ corresponds to $\U_X^{\infty}(n,d)$, which denotes the moduli space of pairs consisting of rank $n$ and degree $d$ vector bundles over $X$ endowed with a formal trivialization at $x\in X$. Thus, considering the long exact sequence of cohomology of:
    $$0 \to \Homc_X(E,E(-mx)) \to \Dcal^1_{X/\C}(E,E(-mx)) \to \Tc_X(-mx)\to 0$$
one obtains:
    \begin{equation}\label{gen:e:Tgs}
      0\to \Tg_{(E,\phi)} \U_X^{\infty}(n,d) \to \Tg_{\Ec}\U_g^{\infty}(n,d) \to \Tg_{(X,x,z)} \M_g^{\infty} \to 0
    \end{equation}
where $\Ec=(X,x,z,E,\phi)$.

In the context of the infinite Grassmannian, the sequence (\ref{gen:e:Tgs}) is (using Proposition \ref{gen:p:tgUGr}):
    {\small$$
    0\to \Hom_{A_W}(W,V/W) \to \Dcal^1_{A_W/\C}(W,V/W) \to \Der_{\C}(A_W,\C((z))/A_W)\to 0\,.
    $$}
The surjective arrow is the symbol map.

%%%%%%%%%%%%%%%%%%%%%%%%%%%%%%%%%%%%%%%%%%%%%%%%%%%%%%%%%%%%%%%%%%%%%%%%%%%%%
\subsection{The group $\SGl_{\C((z))}(V)$ as local generator for $\U_g^{\infty}(n,d)$.}\label{subsect:GroupSGLLocalGen}

\begin{thm}\label{thm:SGLgeneratesUinfty}
The group $\SGl_{\C((z))}(V)$ acts on $\U_g^{\infty}(n,d)$ and this action is locally transitive.
\end{thm}

\begin{proof}
First, note that the action of $\SGl_{\C((z))}(V)$ on $V$ gives rise to an action of  $\SGl_{\C((z))}(V)$ on $\Gr^{\chi}(V)$. Recall (Theorem \ref{gen:t:carU}) that rational points of $\U_g^{\infty}(n,d)$ correspond the to points $W\in \Gr^{\chi}(V)$ such that $A_W$ belongs to $\Gr^{1-g}\big(\C((z))\big)$ and $A_W$ is regular.

We must first check that $\SGl_{\C((z))}(V)$ acts on $\U_g^{\infty}(n,d)$; this means that for all $\gamma \in \SGl_{\C((z))}(V)$ and all $W\in \U_g^{\infty}(n,d)$ it holds that  $\gamma(W)\in \Gr^{\chi}(V)$ and $A_{\gamma(W)}\in \Gr^{1-g} \big(\C((z))\big)$.  It is straightforward to see that $\gamma(W)\in \Gr^{\chi}(V)$.  Moreover,
    {\small\begin{align*}
    A_{\gamma(W)}&=\{a\in \C((z)) \,\, | \,\, a\cdot \gamma(W) \subseteq \gamma(W)\}=\{a\in \C((z))\,\, | \,\, \gamma^{-1}(a \cdot \gamma(W))\subseteq W\}=\\
    &=\{a\in \C((z))\,\, | \,\, g(a)\cdot W\subseteq W\}=g(A_W)
    \, ,
    \end{align*}}
where $g\in \G$ satisfies $\gamma(zv)=g(z)\gamma(v)$ (see Definition \ref{gen:d:SGl}). The conclusion follows since $g(A_W)\in \Gr^{1-g} \big(\C((z))\big)$ (see \cite[Theorem 4.9]{MP2}) and $g(A_W)$ is regular.

In order to see that the action of $\SGl_{\C((z))}(V)$ is locally transitive, it suffices to prove (see \cite{MP2}) that the orbit morphism is surjective at the level of tangent spaces; that is, we have to check that:
    $$\sgl_{\C((z))}(V)\to \Tg_{\Ec} \U_g^{\infty}(n,d)$$
is surjective for all $\Ec=(X,x,z,E,\phi) \in \U_g^{\infty}(n,d)$. Indeed, let us consider the following exact sequence:
    {\small $$
    0\to \Dcal^1_{X/k}(E,E(-mx)) \to \Dcal^1_{X/k}(E,E(\bar mx)) \to \Dcal^1_{X/k}(E,(\Oc_X(\bar mx)/\Oc_X(-mx))^n)\to 0
    \, .
    $$}
Taking cohomology, $\varprojlim_m$ and $\varinjlim_{\bar m}$ we obtain:
    {\small
    \begin{equation}\label{loc:e:genlocU}
    \begin{aligned}
      0 \to \Hcoh^0(X-x,\Dcal^1_{X/\C}(E,E)) \to & \varinjlim_{\bar m} \varprojlim_{m} \Dcal^1_{X/k}(E,(\Oc_X(\bar mx)/\Oc_X(-mx))^n)   \to  \\
    & \to \varprojlim_m \Hcoh^1(X,\Dcal^1_{X/\C}(E,E(-mx))) \to 0
    \, ,
    \end{aligned}
    \end{equation}}
where $\varprojlim \Hcoh^1(X,\Dcal^1_{X/\C}(E,E(-mx)))\iso \Tg_{\Ec} \U_g^{\infty}(n,d)$ by Theorem \ref{gen:t:tgUinfty}.

Notice that:
    {\small $$\varinjlim_{\bar m} \varprojlim_{m} \Dcal^1_{X/k}(E,(\Oc_X(\bar mx)/\Oc_X(-mx))^n) \,=\,  \Dcal^1_{\C((z))/\C}(V,V) \,=\, \sgl_{\C((z))}(V)\,.$$}
implies the exactness of the sequence:
    {\small$$\xymatrix@R=12pt@C=14pt{
    0 \ar[r] & \Hcoh^0(X-x,\Dcal^1_{X/\C}(E,E))  \ar[r] & \sgl_{\C((z))}(V) \ar[r] & \Tg_{\Ec} \U_g^{\infty}(n,d) \ar[r] &0
    }$$}
and the statement is proved.
\end{proof}

\begin{remark}
There are similar sequences for $\M_g^{\infty}$ and $\U_X^{\infty}(n,d)$ to that of $\U_g^{\infty}(n,d)$ given in equation (\ref{loc:e:genlocU}), see for instance \cite{MP2}. Therefore, one has the following diagram:
    {\small$$\xymatrix@R=12pt@C=14pt{
        &0 \ar[d] & 0\ar[d] & 0\ar[d] & \\
        0 \ar[r] &\Hcoh^0(X-x,\Endc_X E) \ar[d] \ar[r] & \gl_{\C((z))}(V) \ar[d] \ar[r] & \Tg_{(E,\phi)} \U_X^{\infty}(n,d) \ar[d] \ar[r] &0\\
        0 \ar[r] & \Hcoh^0(X-x,\Dcal^1_{X/\C}(E,E)) \ar[d] \ar[r] & \sgl_{\C((z))}(V) \ar[d] \ar[r] & \Tg_{\Ec} \U_g^{\infty}(n,d) \ar[d] \ar[r] &0\\
        0 \ar[r] & \Hcoh^0(X-x,\Tc_X) \ar[d]\ar[r] & \g \ar[d] \ar[r] & \Tg_{(X,x,z)} \M_g^{\infty} \ar[d]\ar[r] &0\\
        & 0 & 0 & 0 &
    }$$}
which is identified (via the Krichever map) with the diagram:
    {\small $$\xymatrix@R=14pt@C=14pt{
        &0 \ar[d] & 0\ar[d] & 0\ar[d] & \\
        0 \ar[r] &\End_{A_W}W \ar[d] \ar[r] & \End_{\C((z))}V \ar[d] \ar[r] & \Hom_{A_W}(W,V/W) \ar[d] \ar[r] &0\\
        0 \ar[r] & \Dcal^1_{A_W/\C}(W) \ar[d] \ar[r] & \Dcal^1_{\C((z))/\C}(V) \ar[d] \ar[r] & \Dcal^1_{A_W/\C}(W,V/W) \ar[d] \ar[r] &0\\
        0 \ar[r] & \Der_{\C}(A_W) \ar[d]\ar[r] & \Der_{\C}\big(\C((z))\big) \ar[d] \ar[r] & \Der_{\C}(A_W,\C((z))/A_W) \ar[d]\ar[r] &0\\
        & 0 & 0 & 0 &
        }$$}
\end{remark}

\begin{remark}
Our approach can be viewed as a generalization of \cite{ACKP} and \cite{MP3}. In fact, if we restrict ourselves to the rank $1$ case and look at the level of tangent spaces, then the relation with the results of \cite{ACKP} becomes apparent. Furthermore, here we are essentially studying bundles on curves, while \cite{MP3} deals with coverings of curves; the techniques and goals are, however, similar and closely related.
\end{remark}

\begin{remark}
In \cite{FaRe} the authors make use of similar techniques in order to prove that, for $n=2,d=1$, there is an isomorphism between the second cohomology group of $\Dcal^1=\lsl_{\C((z))}(V)\rtimes \g$ (where $\lsl_{\C((z))}(V)$ denotes the Lie algebra of the special linear group) and the second rational cohomology group of the moduli space parametrizing $4$-tuples consisting of a curve, a point, a tangent vector and a prank $2$ vector bundles with fixed determinant. The study of this cohomological application for higher rank and arbitrary determinant will be the subject of future research.
\end{remark}

%%%%%%%%%%%%%%%%%%%%%%%%%%%%%%%%%%%%%%%%%%%%%%%%%%%%%%%%%%%%%%%%%%%%%%%%%%%%%%%%%%%%%%%%%%%%%%

\section{A relation in the Picard group of moduli of vector bundles on a family of curves.}\label{loc:s:GloMumE}\qquad

In this section we shall show that formula~(\ref{eq:TorsorsOnSGL}) is of geometric nature. Indeed, we shall give certain line bundles on the moduli space of vector bundles over a family of curves, we shall show that these bundles satisfies a similar relation and, more relevantly, that the infinitesimal behavior of the latter is exactly equation~(\ref{eq:TorsorsOnSGL}).

Let $\M_g$ denote the moduli space of genus $g$ smooth projective curves over the field of complex numbers. Let us denote by $\M_g^0$ the open subscheme, consisting of curves without non-trivial automorphisms (henceforth, it will be assumed that $g>2$).

Let  $p_{\Cc}:\Cc\to\M_g^0$ be the universal curve and let $\U_{\Cc}(n,d)\to\M_g^0$ be the moduli space of rank $n$ degree $d$ semistable vector bundles over $\Cc$ (we direct the reader to \cite[Thm. 1.21]{Sim} and \cite[App. 2]{Muk} as well as \cite[Section 2]{BriMa} and \cite[Thm. 5.3.2]{Bri}). It is well known that, in general, the existence of an universal vector bundle on $\Cc\times_{\M_g^0} \U_{\Cc}(n,d)$  may fail (\cite{MestranoRamanan}). We are thus forced to consider the relative situation.

Let $S$ be a scheme and let  $E$ be a relatively semistable vector bundle on $\Cc \times_{\M_g^0} S$. Recall that $E$ yields a map $S\to \U_{\Cc}(n,d)$. Let $\pi_{\Cc}, \pi$ be the projections of $\Cc \times_{\M_g^0} S$ onto the first and second factors, respectively, and $p$ the composition $S\to \U_{\Cc}(n,d)\to \M_g^0$.

\begin{thm}\label{thm:GeometricFormulaLineBundles}
For any integer number $\beta$ let us consider the line bundle on $S$ defined as follows:
    $$
    \L_{\beta} \,:=\, \Det R^{\bullet}\pi_*(E \otimes \pi_{\Cc}^*\omega^{\otimes\beta})
    \, ,
    $$
where $\omega$ is the dualizing sheaf of $\Cc\to\M_g^0$.

Thus, there is an isomorphism of line bundles over $S$:
    \begin{equation}\label{eq:GeomFormLonU}
    \L_{\beta}\, \overset{\sim}\to \, \L_1^{\otimes\beta}\otimes \L_0^{\otimes(1-\beta)}\otimes p^*\lambda_1^{\otimes 6\beta(\beta-1)}
    \, ,
    \end{equation}
where $\lambda_1$ is the Hodge bundle on $\M_g^0$; that is, $\Det R^{\bullet} (p_{\Cc})_* \omega$.
\end{thm}

\begin{proof}
Let $K$ be an effective divisor associated with $\omega$ and let us consider the following diagram:
    $$
    \xymatrix{
    & &  \Cc \times_{\M_g^0} S \ar[dl]_{\pi_{\Cc}} \ar[dr]^{\pi}
    \\
    K \ar@{^(->}[r]^i \ar[drr]_{p_K} & \Cc \ar[dr]^{p_{\Cc}} & & S \ar[ld]_p
    \\
    & & \M_g^0
    }
    \, ,
    $$
where $p_K$ is a finite and flat morphism of degree $2g-2$. Observe that since $p_K$ is finite there exists a line bundle $N$ on $\M_g^0$ such that $i^* \omega=p_K^*N=i^*p_{\Cc}^*N$ and, therefore, $i^*\omega^{\otimes\beta}=i^*p_{\Cc}^*N^{\otimes\beta}$.
Observe that:
    $$
    i_* i^* \omega^{\otimes(\beta+1)} \overset{\sim}\to
    i_*(i^* p_{\Cc}^* N^{\beta}\otimes i^*\omega)\overset{\sim}\to
    i_* i^* (p_{\Cc}^* N^{\beta}\otimes \omega)
    $$
gives rise to an isomorphism between the cokernels of the following two exact sequences of bundles on $\Cc$:
    $$
    0 \to \omega^{\otimes\beta} \to \omega^{\otimes(\beta+1)} \to i_* i^* \omega^{\otimes(\beta+1)} \to 0
    $$
    $$
    0 \to p_{\Cc}^* N^{\beta} \to p_{\Cc}^* N^{\beta}\otimes \omega \to i_* i^* (p_{\Cc}^* N^{\beta}\otimes \omega) \to 0
    \, .
    $$
Taking the pullback by $\pi_{\Cc}$, tensoring by $E$ and considering the determinant in these exact sequences, one obtains an isomorphism:
    {\small $$
    \L_{\beta+1}\otimes \L_{\beta}^{-1} \, \overset{\sim}\to \,
    \Det R^{\bullet}\pi_*(E\otimes \pi_{\Cc}^*\omega\otimes \pi_{\Cc}^*p_{\Cc}^*N^{\otimes \beta}) \otimes
    \Det R^{\bullet}\pi_*(E\otimes \pi_{\Cc}^*p_{\Cc}^*N^{\otimes \beta})^{-1}
    \, .
    $$}
Bearing in mind that $\pi_{\Cc}^*p_{\Cc}^*N \simeq \pi^*p^*N$ and the properties of the determinant functor, the right hand side is isomorphic to:
    {\small $$
    \begin{aligned}
    \Det R^{\bullet}\pi_*(E\otimes \pi_{\Cc}^*\omega)\otimes p_{\Cc}^*N^{\otimes \beta \cdot \chi(E\otimes\pi_{\Cc}^*\omega)}
    &
    \otimes  \Det R^{\bullet}\pi_*(E)^{-1}\otimes (p_{\Cc}^*N^{\otimes \beta \cdot \chi(E)})^{-1}
    \,\overset{\sim}\to
    \\
    &\overset{\sim}\to \,
    \L_{1}\otimes \L_{0}^{-1} \otimes p_{\Cc}^*N^{\otimes \beta n (2g-2)}
    \, ,
    \end{aligned}$$}
where $\chi(M)$ is the Euler-Poincar\'{e} characteristic of the restriction of the sheaf $M$ to the fibers of $\Cc \times_{\M_g^0} S\to S$. Summing up, we have proved that:
    \begin{equation}\label{eq:Lbeta}
    \L_{\beta+1}\otimes \L_{\beta}^{-1} \, \overset{\sim}\to \,
    \L_{1}\otimes \L_{0}^{-1} \otimes p_{\Cc}^*N^{\otimes \beta n (2g-2)}
    \, .
    \end{equation}

Similar to above, but replacing $E$ by the trivial bundle, one obtains the following isomorphism:
    $$
    p^*\lambda_{\beta+1}\otimes p^* \lambda_{\beta}^{-1} \, \overset{\sim}\to \,
    p^*\lambda_{1}\otimes p^* \lambda_{0}^{-1} \otimes p_{\Cc}^*N^{\otimes \beta (2g-2)}
    \, ,
    $$
where $\lambda_{\beta} := \Det R^{\bullet} (p_{\Cc})_* \omega^{\otimes \beta}$. Mumford's formula, which asserts that $\lambda_\beta\simeq \lambda_1^{\otimes(6\beta^2- 6\beta+1)}$ (see~\cite{Mum2}), applied to the above formula yields the identity:
    $$
    p^* N^{\otimes \beta(2g-2)} \, \overset{\sim}\to \, p^* \lambda_1^{\otimes 12\beta}
    $$

Plugging this into equation~(\ref{eq:Lbeta}), one has:
    $$
    \L_{\beta+1}  \, \overset{\sim}\to \, \L_{\beta} \otimes \L_{1}\otimes \L_{0}^{-1} \otimes p^* \lambda_1^{\otimes 12 n \beta}
    \, .
    $$
Proceeding recursively, one obtains the result:
    $$
    \begin{aligned}
    \L_{\beta+1}  \, & \overset{\sim}\to \, \L_1^{\otimes\beta}\otimes (\L_{0}^{\otimes\beta})^{-1} \otimes \L_{0}
    \otimes p^* \lambda_1^{\otimes \sum_{i=1}^\beta 12 n (i-1)}
    \\
    & \overset{\sim}\to \, \L_{1}^{\otimes\beta}\otimes \L_{0}^{\otimes(1-\beta)}
    \otimes p^* \lambda_1^{\otimes 6n\beta(\beta-1)}
    \, .
    \end{aligned}
    $$
\end{proof}

\begin{remark}
In \cite{Sch} the author proved that both sides of the isomorphism~(\ref{eq:GeomFormLonU}) have the same Chern class and offered an interpretation in terms of the $bc$-system of rank $n$.
\end{remark}

The rest of this section is devoted to showing that formulae~(\ref{eq:TorsorsOnSGL}) and~(\ref{eq:GeomFormLonU}) are actually the same. For this goal, we shall follow the ideas of~\cite{MP2}, where the case of the Mumford's formula was studied.

Let $\U_g^{\infty}(n,d)_{\text{ss}}$ denote the subscheme of $\U_g^{\infty}(n,d)$ consisting of those points where the vector bundle is semistable. Note that there is a natural forgetful morphism:
    $$
    \U_g^{\infty}(n,d)_{\text{ss}} \, \longrightarrow \, \U_{\Cc}(n,d)
    \, ,
    $$
which is a homomorphism of $\M_g^0$-schemes. If no confusion arises, we shall use a dashed arrow $\xymatrix{\U_g^{\infty}(n,d) \ar@{-->}[r] &  \U_{\Cc}(n,d)}$ to denote the forgetful morphism defined only on the subscheme $\U_g^{\infty}(n,d)_{\text{ss}}$.

Continuing with the above notations, recall that $S$ is a scheme and $E$ is a relative semistable vector bundle on $\Cc \times_{\M_g^0} S$ that yields a map $S\to \U_{\Cc}(n,d)$. Let us now  assume that a lift of $S\to \U_{\Cc}(n,d)$ to $\U_g^{\infty}(n,d)_{\text{ss}}$ is given; in other words, that one has a $5$-tuple of the type $(\Cc \times_{\M_g^0} S, x,z,E, \phi)$ or, equivalently, a commutative diagram:
    {\small $$
    \xymatrix{
    & \U_g^{\infty}(n,d)   \ar@{-->}[d]
    \\
    S \ar[ru] \ar[r] & \U_{\Cc}(n,d)
    }$$}
Let us also assume that a $\C$-valued point of $S$ is given and let $U\in\U_g^{\infty}(n,d) $ be the rational point of $\Gr(V)$ associated with it through the Krichever map. Using the results of subsection~\ref{subsect:GroupSGLLocalGen}, we have the following diagram:
    {\small $$
    \xymatrix@C=22pt{
    \SGl_{\C((z))}(V) \simeq \SGl_{\C((z))}(V) \times\{U\} \ar[r]^-{\mu_U} \ar@{-->}[d]_{\bar\mu_U} &
    \U_g^{\infty}(n,d)  \ar@{-->}[d] \ar[rd] &
    \\
    S \ar[ru]  \ar[r]_{\bar p} & \U_{\Cc}(n,d) \ar[r] & \M_g^0
    }$$}
We will assume that the orbit of $U$ under the action of $\SGl_{\C((z))}(V)$ falls inside $S$. Thus, $\mu_U$ factors through $S$ and we obtain $\bar\mu_U$ (see the dashed arrow).

We now have the following result, which sheds light on the infinitesimal behavior of formula~(\ref{eq:GeomFormLonU}) at the point $U$:

\begin{thm}\label{thm:pullbackGeomInf}
The pullback of formula~(\ref{eq:GeomFormLonU}):
    $$
    \L_{\beta}\, \overset{\sim}\to \, \L_1^{\otimes\beta}\otimes \L_0^{\otimes(1-\beta)}\otimes p^*\lambda_1^{\otimes 6\beta(\beta-1)}
    $$
by $\bar\mu_U$ is precisely formula~(\ref{eq:TorsorsOnSGL}):
    $$
    {\mathbb L}_{n,\beta}\, \simeq
    \, {\mathbb L}_{n,1}^{\otimes\beta} \otimes {\mathbb L}_{n,0}^{\otimes(1-\beta)} \otimes p^* \Lambda_1^{\otimes 6n\beta(\beta-1)}
    $$
\end{thm}

\begin{proof}
With the above assumptions and notations, it suffices to prove that:
    $$
    \begin{gathered}
    \bar \mu_U^*\L_{\beta}\,\simeq \, {\mathbb L}_{n,\beta}
    \\
    \bar\mu_U^* p^* \lambda_{\beta} \, \simeq \,  \Lambda_{\beta}
    \, .
    \end{gathered}
    $$

Recall from subsection~\ref{loc:s:central} the definition of the sheaves ${\mathbb L}_{n,\beta}$ and $\Lambda_{n}$. Let us denote by $\Kr_{n,\beta}$ the modification of the Krichever map that sends $(X,x,z,E,\phi)$ to $H^0(X-\{x\}, E\otimes\omega_X^{\otimes\beta})$. Thus, from the commutative diagram:
    {\small $$
    \xymatrix@C=26pt{
    \SGl_{\C((z))}(V) \simeq \SGl_{\C((z))}(V) \times\{U\} \ar[r]^-{\mu_U} \ar[d]_{\bar\mu_U} &
    \U_g^{\infty}(n,d)   \ar@{^(->}[r]^-{\Kr_{n,\beta}}  \ar@{-->}[d] &  \Gr(V)
    \\
     S \ar[ru] \ar[r]_{\bar p}   & \U_{\Cc}(n,d)
    }$$}
and the base-change property of the determinant, it is straightforward to see that:
    $$
    {\mathbb L}_{n,\beta} \, \simeq \, \mu_U^*\Kr_{n,\beta}^*\Det \,\simeq \, \bar \mu_U^*\L_{\beta}
    \, .
    $$
The second isomorphism follows analogously and the Theorem is proved.
\end{proof}

\begin{remark}
In cases where a universal or Poincar\'{e} bundle exists on $\Cc\times_{\M_g^0} \U_{\Cc}(n,d)$ (see~\cite{MestranoRamanan}), one could repeat the above construction for $S= \U_{\Cc}(n,d)$ and $E$ the universal object. Accordingly, Theorem~\ref{thm:GeometricFormulaLineBundles} is indeed an identity that holds on the the Picard group of $\U_{\Cc}(n,d)$ (see~\cite{Kou} for some facts on generators of this group).
\end{remark}

\bibliographystyle{siam}
\bibliography{biblio}
\end{document}